\documentclass[11pt]{article}
\usepackage{amsmath}
\usepackage{amsfonts}
\usepackage{a4wide}
\usepackage{amsthm}
\theoremstyle{theorem}
\newtheorem{theorem}{Theorem}[section]
\newtheorem{lemma}[theorem]{Lemma}
\newtheorem{corollary}[theorem]{Corollary}
\newtheorem{proposition}[theorem]{Proposition}
\theoremstyle{definition}
\newtheorem{remark}[theorem]{Remark}
\newtheorem{example}[theorem]{Example}
\newtheorem*{acknowledgements}{Acknowledgements}

\newcommand{\GR}{{\mathnormal{\bf R}}}
\newcommand{\GQ}{{\mathnormal{\bf Q}}}
\newcommand{\GN}{{\mathnormal{\bf N}}}
\newcommand{\GS}{{\mathnormal{\bf S}}}

\newcommand{\cM}{{\mathcal{M}}}
\newcommand{\cO}{{\mathcal{O}}}
\newcommand{\cov}{\mathbb{C}\text{ov}}
\newcommand{\Lip}{\mathbb{L}\text{ip\,}}
\newcommand{\eps}{\epsilon}

\newcommand{\E}{\mathbb{E}}

\renewcommand{\P}{\mathbb{P}}

\def\1{\mathbf{1}}

\title{Phantom distribution functions for some stationary sequences\thanks{ This work has been developed within the MME-DII center of excellence (ANR-
11-LABEX-0023-01).}
}

\author{Paul Doukhan\\ Universit\'e de  Cergy-Pontoise\\ and Institut Universitaire de France \and 
Adam Jakubowski\\  Nicolaus Copernicus University \and Gabriel Lang \\ AGROPARISTECH }

\date{}

\begin{document}
\maketitle

\begin{abstract}
The notion of a phantom distribution function (phdf)  was introduced by \cite{OBr87}. We show 
that the existence of a phdf is a quite common phenomenon for stationary weakly dependent sequences. It is proved that any $\alpha$-mixing stationary sequence with continuous 
marginals admits a continuous phdf. Sufficient conditions are given for stationary sequences exhibiting weak dependence, what allows the use of attractive models beyond mixing. The case of discontinuous marginals is also discussed for $\alpha$-mixing. 

Special attention is paid to examples of processes which admit a continuous phantom distribution function while their extremal index is zero. We show that  \cite{Asmu98} and \cite{RRSS06} provide natural examples of such processes. We also construct a non-ergodic stationary process of this type.

\noindent {\em Keywords:} Strictly stationary processes; Extremes; Extremal index; Phantom distribution function; $\alpha$-mixing; Weak dependence; Lindley's process; Random walk Metropolis algorithm.

\noindent{\em MSC:} 60G70, 60G10, 60F99.
\end{abstract}

\section{Phantom distribution functions}
The notion of a phantom distribution function was introduced in \cite{OBr87}.
Let $\{X_j\}$ be  a stationary sequence with partial maxima 
\[ M_n = \max_{1 \leq j \leq n} X_j\]
and the marginal distribution function $F(x) = \P( X_1 \leq x)$.

 A stationary sequence $\{X_n\}$ is said to admit a phantom distribution function $G$ if
\[
\P(M_n \leq u_n) - G^n(u_n) \to 0, \text{ as $n\to\infty$},\]
for every sequence $\{u_n\}\subset \GR$.  
Since $u_n$ is arbitrary, the above can be written as 
\begin{equation}\label{e1}
\sup_{u\in\GR} \big| \P( M_n \leq u) - G^n(u) \big|\to 0, \text{ as $n\to\infty$}.
\end{equation}
It is obvious that $G$ is not uniquely determined 
for only the behavior of $G$ at its right end $G_* = \sup\{ x\,;\, G(x) < 1\}$ is of importance.
On the other hand, any two phantom distribution functions cannot be too different - we show in Theorem \ref{thm:equiv} below that they must be tail equivalent. 
 
When (\ref{e1}) is satisfied with $G(x) = F^{\theta}(x)$, for some $\theta \in (0,1]$, then 
we say that  $\{X_j\}$ has the extremal index $\theta$ in the sense of  \cite{Lead83} (see also \cite{LLR83}). The notion of the extremal index is well-understood and had been intensively investigated in 1980s and 1990s. 

Another well-known area where  phantom distribution functions naturally occur is when $\{X_j\}$ has a regenerative structure (see e.g. \cite{Asmu03}).
For such processes Theorem 3.1 of \cite{Root88} provides sufficient conditions  for a suitable power of the distribution function of the maximum over the regeneration cycle to be a phantom distribution function of the original sequence. Using this result \cite{Asmu98} exhibited an example of a Markov chain (in fact: the Lindley process with subexponential step distribution) which admits a non-trivial phantom distribution function and has the extremal index $\theta =  0$ in the sense of \cite{Lead83}, that is
\begin{equation}\label{eq: exzero}
 \P\big( M_n \leq u_n(\tau)\big) \to 1
\end{equation}
whenever $\{u_n(\tau)\}$ is such that 
\begin{equation}\label{eq: uentau}
 n (1-F(u_n(\tau)) \to \tau \in (0,+\infty).
\end{equation}
Intuitively this means that partial maxima $M_n$ increase much slower comparing with the independent case
and that information on $F$ cannot determine the limit behavior of laws of $M_n$.  

It follows from \cite{Asmu98} that the existence of a phantom distribution function can remain informative while the extremal index does not contribute anything. 
In the present paper we develop the theory of phantom distribution functions by showing that their existence is a quite common phenomenon among weakly dependent sequences and that in practically all cases of interest we can find a {\em continuous}  phantom distribution function. We also find two further examples of sequences with extremal index $\theta = 0$ and a continuous phantom distribution function.

In Section \ref{sect_gen} we provide convenient necessary and sufficient  conditions for the 
existence of
a {\em continuous} phantom distribution function.  This is done in Theorem \ref{th1}, which is an improvement of 
 \cite{OBr87}, \cite{Jak91} and \cite{Jak93}. According to this theorem we need only to find a {\em single} sequence $\{v_n\}$ of levels such that 
$\P \big( M_n \leq v_n\big) \to \gamma$, for some $\gamma \in (0,1)$, and Condition $B_{\infty}(v_n)$ holds (see (\ref{e4})), which is a form of ``mixing'' specific for maxima. We demonstrate how weak are these requirements by showing in  Theorem \ref{th2} that there exists a non-ergodic stationary sequence (in fact: exchangeable) with the extremal index $\theta = 0$ and admitting a continuous phantom distribution function.

Section \ref{sect_reg} contains a result built upon the extra information we are given when the process has a regenerative structure (Theorem \ref{thm:reg}), which is essentially a version of the mentioned Theorem 3.1. in \cite{Root88}.  We then discuss the motivating Assmussen's example and derive the existence of a continuous phantom distribution function for it.

In Section \ref{sect_list} we list all our results on existence of phantom distribution functions obtained from information on mixing and properties of marginal distributions only. Such conditions are usually not difficult to verify (as showed in Example 2) and therefore widely applicable.
 
In Theorem \ref{thm:alpha} we  prove a remarkable fact that {\em every} strongly  (or $\alpha$-) mixing stationary sequence with {\em continuous} marginals admits a continuous phantom distribution function. Applying this result we are able to show that  
the process constructed in Section 3 of \cite{RRSS06} (random walk Metropolis algorithm for distributions with heavy tails), which has the extremal index zero, also admits a continuous phantom distribution function. Since Metropolis Markov chains are easy in simulation, they can be considered as  a class of reference  processes for the extremal index zero processes in the sense of the relative extremal index defined in \cite{Jak91}.   

Since the paper by \cite{Andr84} it is known that some of time series considered in econometric modeling are non-strong mixing (see also \cite{DDL07}, Section 1.5). Taking this into account, the notion of ``weak dependence coefficients'' has been developed during the last twenty years and the resulting theory is reach enough to cover the most interesting cases.  We refer to \cite{DDL07} for a comprehensive presentation of this trend in the analysis of time series. 

In the present paper we show that the approach through ``weak dependence coefficients" is applicable to our problem, as well. The reasoning is not automatic, since verification of Condition
$B_{\infty}(v_n)$ requires approximation of indicator functions by functions exhibiting more ``smoothness''. Therefore to cope efficiently with  $\theta$-, $\eta$-, $\kappa$- and $\lambda$-weakly dependent sequences we need a bit more regularity than just the continuity of $F$ (see (\ref{Concentration})). Moreover, in Theorems \ref{PHDFtheta} - \ref{PHDFkappa} 
we have to assume that the corresponding weak dependence coefficient converges to zero at some rate. This complies with the general philosophy of weak dependence coefficients, since ``weak dependence' is preserved under sufficiently regular  transformations of processes, with possible change in the rate of decay of the corresponding coefficient (see \cite{DDL07}, Propositions 2.1 and 2.2).  

We conclude our considerations by proposing in Theorem \ref{th3b} how to deal with  a {\em discontinuous} distribution function $F$, for the time being only for $\alpha$-mixing sequences. We show that $m$-dependent sequences admit a continuous phantom distribution function if the marginal distribution function $F$ satisfies only (\ref{e2}) (as in the i.i.d. case), while the non-trivial $\alpha$-mixing seems to require more regularity than (\ref{e2}), even for exponential rate of $\alpha$-mixing.

Section \ref{sect_proofs} contains proofs of all results involving mixing or weak dependence.
The proofs consist in  using our basic  Proposition 
\ref{Propbasic} in order to deduce the rate of convergence $\P(X_1 > v_n) \to 0$, and then  checking Condition $B_{\infty}(v_n)$.

\section{Existence of continuous phantom distribution functions}\label{sect_gen}

 It is an observation made long time ago by \cite{OBr74} (Theorem 2), that for a given distribution function $G$  there exists $\gamma \in (0,1)$ and a sequence $\{v_n = v_n(\gamma)\}$ such that 
\[G^n(v_n) \to \gamma,\]
if, and only if, $G$ satisfies the relations
\begin{equation}
\label{e2}
G(G_*-)=1\quad\text{and}\quad\lim\limits_{x\to G_*-}\dfrac{1-G(x-)}{1-G(x)}=1.
\end{equation}
We will say that $G$ is regular (in the sense of O'Brien) if (\ref{e2}) holds. Notice that if $G$ is regular then the sequence $\{v_n(\gamma)\}$ exists for every $\gamma \in (0,1)$ and that $\{v_n\}$ can always be chosen non-decreasing.

The tail equivalence is another very old notion, introduced by \cite{Res71} and usually
considered in the context of domains of attraction of extreme value distributions. We will modify it slightly, by saying that the tails of two distribution functions $G$ and $H$ with right ends $G_*$ and $H_*$ are {\em strictly tail-equivalent} if 
\begin{equation}\label{eqn:tailequiv}
G_* = H_* \quad\text{and}\quad  \frac{1 - H(x)}{1 - G(x)} \to 1, \text{ as  $x \to G_*-$.} 
\end{equation}

We have a nice characterization of strict tail-equivalence in terms of being mutual phantom 
distribution function. 
\begin{proposition}\label{prop:tailequiv}
Let $G$ be a regular distribution function. Then for any distribution function $H$ the following conditions are equivalent:
\begin{description}
\item{(i)} $H$ is regular and strictly tail-equivalent to $G$.
\item{(ii)}
There exist $\gamma \in (0,1)$ and a non-decreasing sequence $\{v_n\}$ such that 
\begin{equation}\label{eq:tailequiv3}
G^n(v_n) \to \gamma,\quad F^n(v_n) \to \gamma.
\end{equation}
\item{(iii)}
\begin{equation}\label{eq:tailequiv2}
 \sup_{x\in\GR} \big| G^n(x) - H^n(x) \big| \to 0, \text{ as $n\to\infty$}.
\end{equation}
\end{description}
\end{proposition}

\begin{proof} We shall prove $\text{(i)} \iff \text{(ii)}$
first and then $\text{(iii)} \iff \text{(i)}$.

So assume (i). Since $G$ is regular, there are $\gamma \in (0,1)$ and a non-decreasing sequence $\{v_n\}$ such that $G^n(v_n) \to \gamma$.  By the strict tail-equivalence 
\begin{align}
 G^{n}(v_n) &= \exp\big( - n (1-G(v_n))\big) + o(1) \notag \\
&=  \exp\big( - n (1- H(v_n)\big) ^{(1-G(v_n))/(1 - H(v_n))} + o(1) \\
& = \exp\big( - n (1-H(v_n))\big) + o(1) \notag \\
& = H^{n}(v_n) + o(1). \notag
\end{align}
Hence $H^n(v_n) \to \gamma$ and (ii) follows.  

To prove that (ii) implies (i), take suitable $\gamma$ and $\{v_n\}$ and consider a sequence $x_n \nearrow G_*$. Define numbers $m_n$ by
\[v_{m_n} \leq x_n < v_{m_n + 1}.\]
Clearly, $m_n \to \infty$ and we have
\[
1 - G(v_{m_n + 1}) \leq 1 - G(x_n) \leq 1 - G(v_{m_n}).\]
Consequently, 
\[ m_n\big( 1 - G(v_{m_n + 1})\big)  \leq m_n\big(1 - G(x_n) \big) \leq m_n\big(1 - G(v_{m_n})\big)\]
and the first and the third terms go to $-\log \gamma$, hence also 
$m_n\big(1 - G(x_n) \big) \to -\log \gamma$. The same sandwiching holds for 
expressions involving $H$ and so
\[
\lim_{n\to\infty} \frac{1 - G(x_n)}{1 - H(x_n)}   = \lim_{n\to\infty} \frac{m_n\big(1 - G(x_n)\big)}{m_n\big(1 - H(x_n)\big)} = \frac{-\log \gamma} {-\log \gamma} = 1.\]
Since $x_n \nearrow G_*$ was arbitrary, the strict tail equivalence follows. 

Now assume that  (\ref{eq:tailequiv2}) holds. By the regularity of $G$ there exist a number $\gamma \in (0,1)$ and  a sequence of levels $\{v_n(\gamma)\}$ such that $G^n(v_n(\gamma)) \to \gamma \in (0,1)$. Hence also 
$H^n(v_n(\gamma)) \to \gamma \in (0,1)$ and the regularity of $H$ follows.

Next suppose that $G_* < H_*$. Then there exists $x_0 \in \GR$ such that $G(x_0) = 1$, while 
$ H(x_0) < 1$ and 
\[  G^n(x_0) - H^n(x_0) \to 1,\]
what contradicts (\ref{eq:tailequiv2}). Hence $G_* \geq H_*$ and by the symmetry $G_* = H_*$.
 
Now let us take any sequence $x_n \to G_*-$, some $\tau > 0$ and define
\[ m_n = \min\{ m\,;\, m ( 1 - G(x_n)) \geq \tau \}.\]
Since $1 - G(x_n) \to 0$, we have $m_n (1 - G(x_n)) \to \tau$, hence also $G^{m_n}(x_n) \to \exp(-\tau)$. By (\ref{eq:tailequiv2}) $H^{m_n}(x_n) \to \exp(-\tau)$ and so $m_n (1 - H(x_n)) \to \tau$. Finally we have
\[  \frac{1 - G(x_n)}{1 - H(x_n)} = \frac{m_n \big(1 - G(x_n)\big)}{m_n \big(1 - H(x_n)\big)}
\to \frac{\tau}{\tau} = 1,\ \text { as $n\to\infty$.}
\]

To prove (i) $\Longrightarrow $ (iii), let us assume that $H$ is regular and strictly tail-equivalent to $G$, but 
(\ref{eq:tailequiv2}) does not hold, i.e. there is a subsequence $\{n_k\} \subset \GN$, a sequence $\{x_{k}\} \subset \GR$ and $\eta > 0$ such that 
\[ \big|G^{n_k}(x_k) - H^{n_k}(x_k)\big| > \eta, \ k \in \GN.\]

Suppose that along some further subsequence $\{n_{k_l}\}$ we have
 $G(x_{k_l}) \leq 1 - \delta$, for some $\delta > 0$. Then $G^{n_{k_l}}(x_{k_l}) \to 0$ and 
 $H^{n_{k_l}}(x_{k_l}) \geq \eta$ for $l$ large enough. This gives $\sup_l n_{k_l}\big(1 - H(x_{k_l})\big) < +\infty$, hence $H(x_{k_l}) \to 1$ and $x_{k_l} \to H_*- = G_*-$. This in contradiction  with $G(x_{k_l}) \leq 1 - \delta$. 

So we may and do assume that $x_k \to G_*-$. Then by the strict tail-equivalence 
\begin{align*}
 G^{n_k}(x_k) &= \exp\big( - n_k (1-G(x_k))\big) + o(1) \\
&=  \exp\big( - n_k (1- H(x_k)\big) ^{(1-G(x_k))/(1 - H(x_k))} + o(1) \\
& = \exp\big( - n_k (1-H(x_k))\big) + o(1) \\
& = H^{n_k}(x_k) + o(1).
\end{align*}
We have once again arrived to a contradiction, this time with the choice of $\{n_k\}$ and $\{x_k\}$. 
\end{proof}

The above proposition yields immediately the following useful fact.
\begin{theorem}\label{thm:equiv}
Suppose that a stationary sequence $\{X_j\}$ admits a {\em regular} phantom distribution function $G$. 

Let $H$ be any other phantom distribution function for   $\{X_j\}$.
Then $H$ is also regular and $G$ and $H$ are strictly tail-equivalent.

Conversely, if $H$ is regular and strictly tail-equivalent to $G$, then it is also a phantom distribution function function for $\{X_j\}$.
\end{theorem}

Somewhat surprisingly, it is possible to provide a complete description 
of stationary sequences admitting a 
phantom distribution function  and the description is in terms of natural and verifiable conditions.  The following theorem slightly improves the results of  \cite{OBr87}, \cite{Jak91} and \cite{Jak93}, since we construct a {\em continuous} phantom distribution function.

\begin{theorem}\label{th1}
Let $\{X_j\}$ be stationary. The following are equivalent:
\begin{description}
\item{\em (i)} The sequence $\{X_j\}$ admits a {\em continuous} phantom distribution function.
\item{\em (ii)} The sequence $\{X_j\}$ admits a {\em regular} phantom distribution function.
\item{\em (iii)}  There exists a sequence $\{v_n\}$ and $\gamma \in (0,1)$  such that 
\begin{equation}\label{e3}
\P(M_n\leq v_n)\to \gamma,
\end{equation}
and the following Condition $B_{\infty}(v_n)$ holds: 
\begin{equation}\label{e4}
\sup_{p,q\in \GN}\left| \P\big(M_{p+q}\leq v_n\big) - \P\big(M_p\leq v_n\big) \P\big(M_q\leq v_n\big)\right| 
\to 0,\ \text{ as $n\to\infty$.}
\end{equation}
\item{\em (iv)}  There exists a sequence $\{v_n\}$ and $\gamma \in (0,1)$  such that (\ref{e3}) holds 
and for each $T > 0$ the following Condition $B_{T}(v_n)$ is fulfilled: 
\begin{equation}\label{e4a}
\sup_{p,q\in \GN,\atop \ p+q \leq T\cdot n}\left| \P\big(M_{p+q}\leq v_n\big) - \P\big(M_p\leq v_n\big) \P\big(M_q\leq v_n\big)\right| 
\to 0,\ \text{ as $n\to\infty$.}
\end{equation}
\item{\em (v)} There exists a sequence $\{v_n\}$ and $\gamma \in (0,1)$ such that for some dense subset $\GQ\subset \GR^+$
\begin{equation}\label{e5}
\P\big(M_{[nt]}\leq v_n\big)\to \gamma^t,\ t\in\GQ.
\end{equation}
\end{description}
If $\{v_n\}$ is strictly increasing, then a continuous phantom distribution function 
$G$ can be defined by  
\[ G(x) = \gamma^{g(x)},\]
where
\begin{equation}\label{ephdf}
g(x)=\begin{cases} v_1 - x + 1,&\text{if $x<v_1$,}\\
\displaystyle\frac{- x + (n+1) v_{n+1} - n v_n}{n(n+1)(v_{n+1} - v_n)},&\text{if $v_n\le x<v_{n+1}$,}\\
0,&\text{if $x\ge \sup\{v_n\,:\,n\in \GN\}$}
\end{cases}
\end{equation}
If $\{v_n\}$ is not strictly increasing but non-decreasing only, than a slightly more complicated formula for $G$ is given in (\ref{eq:cphdf}) below.  
\end{theorem}

\begin{proof} Theorem 1.3 of \cite{Jak91} establishes the equivalence of (ii) and (iii). Proposition 2.5 of \cite{Jak91} gives the equivalence of (iii) and (iv). Finally, Theorem 2 and Corollary 5 of \cite{Jak93} prove the equivalence of (ii) an (v). In particular, Theorem 1.3 of \cite{Jak91} or Corollary 5 of \cite{Jak93}  lead to a formula for a {\em discontinuous} phantom distribution function 
\begin{equation}\label{ephdfdisc}
\widetilde{G}(x)=\begin{cases} 0,&\text{if $x<v_1$,}\\
\gamma^{1/n},&\text{if $v_n\le x<v_{n+1}$,}\\
1,&\text{if $x\ge \sup\{v_n\,:\,n\in \GN\}$},
\end{cases}
\end{equation}
where $v_n$ is a {\em non-decreasing} sequence obtained in a simple way from the original
one (see Lemma 1 in \cite{Jak93}). For non-decreasing $\{v_n\}$ we can have that
\begin{align*}
 v_1 = v_2 = \cdots = v_{p_1} &< v_{p_1+1} = v_{p_1 + 2} = \cdots = v_{p_2} \\
& < v_{p_2+1} = v_{p_2 + 2} = \cdots = v_{p_3} \\
&\vdots \\
& < v_{p_k+1} = v_{p_k+2} =  \cdots = v_{p_{k+1}} < v_{p_{k+1}+1} \ldots,
\end{align*}
 for some sequence $p_1 < p_2 < \ldots$. Notice that $\widetilde{G}$ jumps only at points 
$v_{p_k}, k = 1,2,\ldots$ and the other elements of the sequence $\{v_n\}$ are not used in the construction of $\widetilde{G}$. We then construct a {\em continuous} phantom distribution function by setting $G(x) = \gamma^{g(x)}$, where
 \begin{equation}\label{eq:cphdf}
g(x)=\begin{cases} v_{p_1} - x + 1/p_1,&\text{if $x<v_{p_1}$,}\\
1/p_k, &\text{if $x= v_{p_k}$,}\\
\text{linearly} &\text{between $v_{p_{k}}$ and $v_{p_{k+1}}$}, k = 1,2, \ldots\\
0,&\text{if $x\ge \sup\{v_n\,:\,n\in \GN\}$}
\end{cases}
\end{equation}
By the very definition we have 
\begin{align*}
\widetilde{G}(x) = 0 < G(x) \leq \widetilde{G}(v_{p_1}), &\text{ if $x < v_{p_1}$},\\
\widetilde{G}(v_{p_k}) = \widetilde{G}(x)  \leq  G(x) < \widetilde{G}(v_{p_{k+1}}), 
 &\text{ if $v_{p_k} \leq x < v_{p_{k+1}}$}, \\
\widetilde{G}(x) = 1 = G(x), & \text{ if $x \geq \sup \{v_{n}\,;\, n \in \GN\}$}.
\end{align*}
We have to prove that
\begin{equation}\label{eqqu} G^n(u_n) - \widetilde{G}^n(u_n) \to 0,
\end{equation}
for every sequence $u_n \in \GR$, such that $u_n < \sup \{v_{n}\,;\, n\in \GN\}$.

 Set $v_{p_0} = -\infty$ and let $k_n\in\GN$ be such  that
\[ v_{p_{k_n - 1}}\leq u_n < v_{p_{k_n}},\ n\in\GN.\]
If $k_{n'} = 1$ along a subsequence $n'$, then
\[ \big|G^{n'}(u_{n'}) - \widetilde{G}^{n'}(u_{n'})\big| = G^{n'}(u_{n'}) < \gamma^{n'/p_1} \to 0.\]
Hence we can assume that  $k_{n} > 1, n\in\GN$. Then
\[
\big|G^n(u_n) - \widetilde{G}^n(u_n)\big| \leq \widetilde{G}^n(v_{p_{k_n}}) - \widetilde{G}^n(v_{p_{k_{n}-1}}) =: R(n).\]
Suppose that along some subsequence $n'$ we have $n' \big(1 - \widetilde{G}(v_{p_{k_{n'}}})\big) \to \infty$. Then 
\[ R(n') \leq \widetilde{G}^{n'}(v_{p_{k_{n'}}}) \to 0.\]
If, on the contrary, $n \big(1 - \widetilde{G}(v_{p_{k_{n}}})\big) \leq K  < \infty$, then
\begin{align*}
R(n) &\leq n\big( \widetilde{G}(v_{p_{k_n}}) - \widetilde{G}(v_{p_{k_{n}-1}}) \big)= 
n \Delta \widetilde{G}(v_{p_{k_n}}) \\
&= \frac{\Delta \widetilde{G}(v_{p_{k_n}})}{
1 - \widetilde{G}(v_{p_{k_n}})} n \big(1 - \widetilde{G}(v_{p_{k_n}})\big) \leq K
\frac{\Delta \widetilde{G}(v_{p_{k_n}})}{
1 - \widetilde{G}(v_{p_{k_n}})} \to 0, 
\end{align*}
for $\widetilde{G}$ is regular and therefore (\ref{e2}) holds.
Hence  (\ref{eqqu}) is always satisfied. 

When $\{v_n\}$ is strictly increasing, (\ref{eq:cphdf}) simplifies to (\ref{ephdf}). The theorem has been proved. 
\end{proof}

\begin{remark}\label{iid}
The above theorem states that an i.i.d. sequence $\{X_j\}$ with regular marginal distribution function $F$ admits a {\em continuous } phantom distribution function $G$. Thus from the point of view of limit theorems for maxima of i.i.d. sequences we can assume that the marginal distributions are continuous. We do not know whether such a reduction is always possible for weakly dependent stationary sequences.

The other consequence is that in every class of strict tail-equivalence of a~regular distribution function $F$ one can find a continuous representative $G$. 
\end{remark}
\begin{remark}\label{super-heavy}
It is important that $F$ {\em does not need to belong to the domain of attraction of any extreme value distribution}, while the corresponding sequence $\{v_n\}$ can have a quite regular form and be suitable for estimation. To see this, let us consider an i.i.d. sequence $\{X_j\}$ with a {\em super-heavy tail}, e.g. 
\[ 1 - F(x) = x^{-1/\sqrt{\log x}}, x > 1.\]
Since for any $c > 0$ we have $x^{-c}/1 - F(x) \to 0$ as $x\to\infty$, there are no normalizing sequences $a_n$ and $b_n$ such that 
\[ \P\big( \max_{1\leq j\leq n} X_j  \leq a_n x + b_n\big) \to G_{\rho}(x), \ x\in \GR,\]
where  $G_{\rho}$ is the standardized extreme value distribution with the extreme value index $\rho \in \GR$ (see \cite{dHFe06}, Theorem 1.1.3).  
On the other hand, if we set $v_n = n^{\log n}$, then $F^n(v_n) \to e^{-1}$.
\end{remark}

\begin{remark}\label{single}
It follows from (iii) or (iv) in the above theorem that the asymptotics of maxima of weakly dependent stationary sequences is {\em fully determined} by the behavior of $\P\big(M_n \leq v_n\big)$ along 
a {\em single} sequence of levels $\{v_n\}$. In particular, if (\ref{e3}) and (\ref{e4}) hold and  
\begin{equation}\label{e6}
F^n(v_n) = \exp\big( - n (1 - F(v_n) ) + o(1) \to \gamma' \in (0,1),
\end{equation}
then $\{X_j\}$ admits a phantom distribution function $G(x) = F^{\theta}(x)$,
where 
\[ \theta = \frac{\log \gamma}{\log \gamma'}\]
is {\em the extremal index}.
Notice the simplified (a single sequence!) form of our approach to the extremal index.

In the next result we shall cover also the case $\theta = 0$ and 
obtain the~final generalization of formula (4.2), p. 380 in \cite{Root88}, which was originally derived for Markov chains with regenerative structure.
\end{remark}

\begin{theorem}\label{ThEI}
Suppose that a stationary sequence $\{X_j\}$ with a regular marginal distribution function $F$ admits a regular phantom distribution function $G$. The following are equivalent:
\begin{description}
\item{\em (i)} The sequence $\{X_j\}$ has the extremal index $\theta \in [0,1]$.
\item{\em (ii)} There exists the limit
\begin{equation}\label{e1EI}
 \lim_{x\to F_*-} \frac{1 - G(x)}{1 - F(x)} (= \theta).
\end{equation}
\item{\em (iii)}  There exist: a sequence $\{v_n\}$, $v_n \nearrow F_*$, and numbers $\gamma \in (0,1)$, $\gamma' \in [0,1)$ such that 
\begin{equation}\label{e1EII}
\lim_{n\to\infty} G^n(v_n) \to \gamma,\quad \lim_{n\to\infty} F^n(v_n) = \gamma'.
\end{equation}
\end{description}
If the limits in (iii) do exist, then $\theta  = \frac{\log \gamma}{\log \gamma'}$ if $\gamma' > 0$ and $\theta = 0$ if $\gamma' = 0$.
\end{theorem}

\begin{proof}
First suppose that the extremal index exists and is equal $\theta \in (0,1]$. This means that $F^{\theta}$ is another phantom distribution function for $\{X_j\}$. By Theorem \ref{thm:equiv} $F^{\theta}$ and $G$ are strictly tail-equivalent. Hence
\[ \lim_{x\to F_*-}\frac{1-G(x)}{1 - F(x)} = \lim_{x\to F_*-}\frac{1 - G(x)}{1 - F^{\theta}(x)}\frac{1 - F^{\theta}(x)}{1 - F(x)} = \theta.\]
Conversely, if the above limit is $\theta$, then $G$ and $F^{\theta}$ are strictly tail-equivalent. Moreover, by (ii) in Proposition \ref{prop:tailequiv} this is equivalent to the existence of $\gamma \in (0,1)$ and a non-decreasing sequence $\{v_n\}$ such that $G^n(v_n) \to \gamma$ and $(F^{\theta})^n(v_n) \to \gamma$, or equivalently $F^n(v_n) \to \gamma^{1/\theta} = \gamma'$. This proves the theorem in the case $\theta \in (0,1]$.

Now suppose that the extremal index of $\{X_j\}$ is $0$. Since we assume that $F$ is regular,  there is a nondecreasing sequence $\{u_n\}$ such that $n\big(1 - F(u_n)\big) \to \tau > 0$. By the definition of the extremal index $\theta = 0$ this implies $G^n(u_n) \to 1$ or $n\big(1 - G(u_n)\big) \to 0$.  Now we follow the proof of Proposition \ref{prop:tailequiv}. Take $x_n \nearrow F_*$ and define numbers $m_n$ by
\[u_{m_n} \leq x_n < u_{m_n + 1}.\]
Then $m_n \big( 1 - F(x_n)\big) \to \tau$ while $m_n \big( 1 - G(u_n)\big) 
\to 0$ and so
\[ \lim_{n\to\infty} \frac{1- G(x_n)}{1 - F(x_n)} = \lim_{n\to\infty} \frac{m_n\big(1- G(x_n)\big)}{m_n\big(1 - F(x_n)\big)}= 0.\]
Since  $x_n \nearrow F_*$ was arbitrary, (\ref{e1EI}) follows.

Now assume (\ref{e1EI}) and suppose that $G^n(v_n) \to \gamma \in (0,1)$ for some sequence $v_n$. If along some subsequence $\{n_k\}$
\[\sup_k n_k \big(1 - F(v_{n_k}\big) \leq M <+\infty,\]
then for $1 > \delta >\gamma$ and large $k$
\begin{align*}
1 > \delta \geq  G^{n_k}(v_{n_k}) &= \exp\big( - n_k (1-G(v_{n_k}))\big) + o(1)  \\
&=  \exp\big( - n_k (1- F(v_{n_k})\big) ^{(1-G(v_{n_k}))/(1 - F(v_{n_k}))} + o(1) \\
& \geq  \exp\big( - C^{(1-G(v_{n_k}))/(1 - F(v_{n_k}))})\big) + o(1)  \to 1.
\end{align*}
This is a contradiction and so $n (1 - F(v_n)) \to \infty$ and (\ref{e1EII}) follows with $\gamma' = 0$.

It remains to show that (\ref{e1EII}) with given $\{v_n\}$, $\gamma \in (0,1)$ and $\gamma' = 0$ implies that the extremal index is $0$. 
First notice that for each $ t\geq 0$ and as $n\to\infty$
\begin{equation}\label{efunct}
G^{[nt]}(v_n) = \big(G^n(v_n)\big)^t + o(1) = \gamma^t + o(1).
\end{equation}
Then observe that $\gamma' =0$ gives 
\begin{equation}\label{eqinfty}
n \big( 1 - F(v_n)) \to \infty.
\end{equation}
 Now suppose that for some $\tau \in (0,+\infty)$
\begin{equation}\label{ejakies}
 n (1 - F(u_n)) \to \tau,
\end{equation}
and along a subsequence $\{n_k\}$
\begin{equation}\label{ebeta}
 \lim_{k\to\infty} G^{n_k}(u_{n_k}) = \beta < 1.
\end{equation} 
Take  $t > 0$ so small that $\gamma^{t} > \beta$. By (\ref{ebeta}) and (\ref{efunct})
\[
 \lim_{k\to\infty} G^{n_k}(u_{n_k}) =\beta < \gamma^{t} = \lim_{k\to\infty} G^{n_k}\big(v_{[n_k/t]}\big).
\]
It follows that eventually $u_{n_k} < v_{[n_k/t]}$, hence by (\ref{eqinfty})
and (\ref{ejakies})
\begin{align*}
+\infty > \tau = \lim_{k\to\infty}  n_k \big(1 - F(u_{n_k})\big) &\geq \limsup_{k\to\infty} n_k \big(1 - F(v_{[n_k/t]})\big)\\
 &\geq \limsup_{k\to\infty} (1/t) 
[n_k/t] \big(1 - F(v_{[n_k/t]})\big) = +\infty.
\end{align*}
 We thus obtained a contradiction and so  $\lim_{n\to\infty} \P\big( M_{n} \leq u_{n} \big) = G^n(u_n) + o(1) = 1$ for any sequence satisfying
(\ref{ejakies}).
\end{proof}

\begin{remark}\label{Betebeinfty}
Condition $B_{\infty}(v_n)$ is considered here for its elegance and concise form. Theorem \ref{th1} shows that it is a synonym for the statement 
``Condition $B_{T}(v_n)$ holds for each $T > 0$''. And it is the  the latter that is checkable in most models, as shown in Section   \ref{sect_proofs}.
\end{remark}

\begin{remark}\label{AIMnot}
\cite{OBr87} introduced Condition AIM$(v_n)$ which was the direct inspiration for our Conditions $B_{\infty}(v_n)$ and $B_T(v_n)$. A stationary sequence $\{X_j\}$ is said to have asymptotic independence of maxima (AIM) relative to a sequence $\{v_n\}$ of real numbers 
if there exists a sequence $r_n$ of positive integers with $r_n = o(n)$ such that 
\[
\max_{p,q,r \geq r_n,\atop \ p+r+q \leq  n}\left| \P\big(M_p \leq v_n, M_{p+r, p+r+q} \leq v_n,\big) - \P\big(M_p\leq v_n\big) \P\big(M_q\leq v_n\big)\right| 
\to 0,
\]
as $n\to\infty$, where $M_{m,n} = \max_{m\leq j \leq n} X_j$. 

We prefer conditions like  $B_{\infty}(v_n)$ and $B_{T}(v_n)$ for three reasons. First, as we could see above, they are necessary and being independent of any separating
sequence $r_n$ are much more convenient in theoretical considerations. Second, finding the proper length of the separating gap $r_n$ does not need to be easy, as the proofs given in Section  \ref{sect_proofs} show. And finally - the name AIM is misleading, for there might be no asymptotic independence at all, as it is demonstrated by the following theorem.
\end{remark}

\begin{theorem}\label{th2}
There exists a stationary sequence $\{X_j\}$ which admits a continuous phantom distribution function, has the extremal index $\theta = 0$ and is non-ergodic.
\end{theorem}

\begin{proof}
We shall construct $\{X_j\}$ as a mixture of i.i.d. sequences. Let $\Omega = \GN \times \GR^{\infty}$ and let $\Pi((k, x_1, x_2, \ldots)) = k$, $X_j((k, x_1, x_2, \ldots)) = x_j$, for $j=1,2,\ldots$. Choose  a strictly increasing sequence $\{v_n\} \in \GR$ and for $k\in \GN$ define a purely jump distribution function $F_k$  by
\[ F_k ( x)  =\begin{cases} 0  & \text{ if  $x < v_{k^2}$},\\
1 - 1/n &\text{ if  $ v_n \leq x < v_{n+1},\ n \geq k^2$}.\end{cases}\]
Now set 
\[ \P(\Pi = k) = \frac{1}{k(k+1)},\quad k = 1,2,\ldots,\]
and define the conditional distribution of $(X_1,X_2,\ldots)$ given $\Pi=k$ as a product $\mu_k \times \mu_k \times \mu_k \times \cdots $, where the probability measure $\mu_k$ corresponds to the distribution function $F_k$.

We have 
\[ \P\big( M_n \leq v_n\big) = \sum_{k=1}^{\infty} \P(\Pi=k) F_k^n(v_n) \]
and for each $k$
\[  F_k^n(v_n) = \exp( - n (1 - F_k(v_n))) + o(1).\]
But
\[ n (1 - F_k(v_n)) = n\Big(  I( k > \sqrt{n}) + (1/n) I( k \leq \sqrt{n})\Big) = 1 \text{ if $n$ is large enough},\]
and so 
\[ \P\big( M_n \leq v_n\big) \to \exp(-1).\]
Similarly, for each $t \geq 0$
\begin{align}\label{eexp}
 \P\big( M_{[nt]} \leq v_n\big) &= \sum_{k=1}^{\infty} \P(\Pi=k) 
F_k^{[nt]}(v_n) \nonumber \\
&=  \sum_{k=1}^{\infty} \P(\Pi=k) \exp( - [nt] (1 - F_k(v_n))) + o(1) \\
& \to  \exp(-t). \nonumber
\end{align}
Hence it follows from Theorem \ref{th1} (v) that $\{X_j\}$ admits a continuous phantom distribution function $G$ given by  (\ref{ephdf}) with $\gamma = e^{-1}$. In particular $G^n(v_n) \to \gamma \in (0,1)$. In view of (iii) in Theorem \ref{ThEI}, $\{X_j\}$ has the extremal index $\theta = 0$
provided  $n \P\big( X_1 > v_n\big) \to \infty$. This is so, indeed.
\begin{align}\label{infinity}
n \P\big( X_1 > v_n\big) &=  \sum_{k=1}^{\infty} \P(\Pi=k)  n (1 - F_k(v_n)) \nonumber \\
&= n \sum_{k=1}^{\infty} \frac{1}{k(k+1)} \Big( 1 I( k > \sqrt{n}) + (1/n) I( k \leq \sqrt{n})\Big) \nonumber\\
&= n \sum_{k = [\sqrt{n}] + 1}^{\infty}  \frac{1}{k(k+1)} + \sum_{k=1}^{[\sqrt{n}]} \frac{1}{k(k+1)} \nonumber \\
& \geq \frac{n}{ [\sqrt{n}] + 1} \to \infty.
\end{align}

To complete the proof let us notice that any set $\{\Pi = k\}$ is invariant for our stationary sequence and so
$\{X_j\}$ is non-ergodic.   
\end{proof}

\section{Phantom distributions for regenerative processes}
\label{sect_reg}
Stochastic processes with regenerative structure provide a natural framework for comparison with {\em some}  i.i.d. sequence. We refer to  Chapters VI and VII in \cite{Asmu03} for a general theory of regenerative processes. Here we shall adopt a minimal formalism, corresponding to limit theorems for maxima. Suppose that there exist integer-valued random variables 
$0 < S_0 < S_1 < S_2 < \cdots $, representing ``regeneration times". Denote by $$W_0= S_0, W_1 = S_1 - S_0, W_2 = S_2 - S_1, \ldots, $$ the length of the consecutive regeneration cycle
and by 
\[ Y_0 = \max_{0 \leq j <  S_0} X_j,\  \ Y_1 = \max_{S_0 \leq j < S_1} X_j, \ \ Y_2 = \max_{S_1 \leq j < S_2} X_j,\ldots \]
the maxima over the regeneration cycles. We assume that 
\begin{equation}\label{eq:reg1}
\begin{aligned}
\big( W_0, Y_0\big), \big( W_1, Y_1\big), \big( W_2, Y_2\big), &\ldots,\text{ are independent,}\\
\big( W_1, Y_1\big), \big( W_2, Y_2\big), &\ldots, \text{ are identically distributed.}
\end{aligned}
\end{equation} 
With this notation we have the following variant of Theorem 3.1 in \cite{Root88}.

\begin{theorem}\label{thm:reg}
If $\mu = \E W_1 < +\infty$ and 
\begin{equation}\label{eq:zerocycle}
\P\big( Y_0 > \max_{1 \leq j \leq n} Y_j\big) \to 0,\ \text{ as $n\to \infty$},
\end{equation}   
then $\{X_j\}$ admits a continuous phantom distribution function if, and only if, 
the distribution function of $Y_1$ is regular. 

In such a case $G(x) = \P^{1/\mu}\big( Y_1 \leq x\big)$ is a regular phantom distribution function for $\{X_j\}$.
\end{theorem}
\begin{proof} \cite{Root88}, p. 375, proves that
\[ \sup_{x\in\GR} \big| \P\big( M_n \leq x\big) - G^n(x)\big| \to 0,\ \text{ as $n\to\infty$}.
\]
If $\P\big(Y_1 \leq x\big)$ is regular, so is $G(x) =  \P^{1/\mu}\big( Y_1 \leq x\big)$, hence $\{X_j\}$ admits a regular  phantom distribution function and by Theorem \ref{th1} also a continuous phantom distribution function. Conversely, if $\{X_j\}$ admits a continuous phantom distribution
$G'(x)$, then by Theorem \ref{thm:equiv} $G(x)$ is regular and so $ \P\big( Y_1 \leq x\big) = G^{\mu}(x)$ 
is also regular. 
\end{proof}

It follows from the above theorem that even if we are given a regenerative structure and are able to check both (\ref{eq:zerocycle}) and $\E W_1 <+\infty$,
we still need additional information in order to obtain the regularity of the distribution of the maximum over the cycle. In formulas (\ref{eq:set}) and (\ref{eq:set1}) below we show how to do that for Lindley's processes with subexponential steps.
   
\begin{example}\label{Ex:Asmussen} (Lindley process)

In general we follow \cite{Asmu98}, but we have changed the notation to comply with the rest of the paper. Let
\[ X_{j+1} = \big( X_j + Z_j\big)^+,\ \  j=1,2,\ldots,\]
where $Z_1, Z_2, \ldots $ are i.i.d. with a distribution function $H$ and mean $-m < 0$ and 
$X_0$ is independent of $\{Z_j\}$ and distributed according to the unique stationary distribution $F$. 
Suppose that $H$ is subexponential, i.e. strictly tail-equivalent to a 
distribution function $B(x)$ concentrated on $(0,\infty)$ and such that   
\[ \frac{1 - B^{*2}(x)}{1-B(x)} \to 2, \text{  as $x\to \infty$}.\]
Then $\{X_j\}$  is a stationary process with regenerative structure $\{(W_j,Y_j)\}$ satisfying (\ref{eq:reg1}) and $\mu = \E W_1 <\infty$. In particular, Theorem \ref{thm:reg} applies to $\{X_j\}$.

Moreover, Theorem 2.1 of \cite{Asmu98} shows that
\[ \frac{\P\big( Y_1 > x\big)}{\mu\big(1-H(x)\big)} \to 1,\ \text{ as $x\to\infty$}.\]
It follows, that $G(x) = \P^{1/\mu}\big( Y_1 \leq x\big)$ is strictly tail-equivalent with $H(x)$.
The regularity of $H$ is implied by the subexponentiality. Indeed, it is well-known (see e.g. 
Lemma 1.3.5 in \cite{EKM97}) that for each $y > 0$
\begin{equation}\label{eq:set}
 \frac{1 - H(x-y)}{1-H(x)} \to 1, \text{ as $x\to\infty$}.
\end{equation}
Hence we have  for any $y>0$
\begin{equation}\label{eq:set1}
 1 \leq \frac{1-H(x-)}{1-H(x)} \leq \frac{1 - H(x-y)}{1-H(x)} \to 1, \text{ as $x\to\infty$}.
\end{equation}

So far we have established that $H$ is a regular phantom distribution function for 
$\{X_j\}$. Notice that this means 
\[ \sup_{x\in\GR}\big| \P\big( \max_{1\leq j\leq n} X_j \leq x\big) -  \P\big( \max_{1\leq j\leq n} Z_j \leq x\big)\big| \to 0, \text{ as $n\to\infty$}.\]
In order to prove that $\{X_j\}$ has the extremal index 0, we can apply our Theorem \ref{ThEI} (ii). Following \cite{Asmu98} let us invoke the 
known asymptotics of the tail of the stationary distribution $F$. By \cite{EmVe82} we have 
\[ 1 - F(x) \sim \frac{1}{m} \int_x^{\infty} \big(1 - H(u)\big) du, \text{ as $x\to\infty$},\]
what is heavier than $1-H(x)$ for subexponential $H$:
\[ \frac{1 - H(x)}{1-F(x)} \to 0,    \text{ as $x\to\infty$}.\]


\end{example}

\begin{remark}\label{rem:Asmussen} 
Lindley process is a simple model which allows (almost) explicit calculations of basic characteristics. In general such situation is very seldom. Therefore the indirect methods 
of construction of a phantom distribution function, presented in the next section and based on mixing Condition $B_{\infty}(v_n)$ and properties of the marginals, seem to be more applicable.  
\end{remark}

\section{Phantom distribution functions for weakly dependent sequences}\label{sect_list}
 \subsection{Coefficients of weak dependence}\label{coefficients}

 The unified contemporary approach to the weak dependence, developed in \cite{DDL07}, consists in establishing specific bounds on covariances between classes of functions. The general framework is as follows. 

For $s\in\GN$, let $\cM^s$ be the family of all bounded measurable {\em non-constant} functions $f$ on $\GR^s$ satisfying
\[ \sup_{(x_1,x_2,\ldots, x_s) \in \GR^s} |f(x_1,x_2,\ldots,x_s)| \leq 1.\]
Define also the Lipschitz coefficient $\cM^s \ni f \mapsto \Lip f \in \GR^+ \cup \{+\infty\}$ as
\[
\Lip f=\sup_{(y_1,\ldots,y_s)\ne(x_1,\ldots,x_s) }
\frac{\left| f(y_1,\ldots,y_s)- f(x_1,\ldots,x_s) \right|}
{\|y_1-x_1\|+\cdots+\|y_s-x_s\|}. \]
Finally, set $\cM = \cup_{s\in \GN} \cM^s$.

In general we will say that a time series $\{X_j\}$ is $\eps$-weakly dependent, if there
exists a mapping 
\[ \Psi_{\eps} : \cM \times \cM \to \GR^+\] such that
 \begin{equation}\label{weakdep}
\eps(r) = \sup \frac{\left|\cov\left(f(X_{i_1},\ldots, X_{i_s}),g(X_{j_1},\ldots, X_{j_t})\right)
\right|}{\Psi_{\eps}(f,g)} \longrightarrow 0, \ \text{ as $r\to\infty$},
 \end{equation}
where the supremum is taken over all pairs of functions $f \in \cM^s$, $g\in\cM^t$ and sets of 
indices 
\[i_1 \le i_2  \le \cdots  \le i_s  \le j_1  \le j_2  \le \cdots  \le j_t\]
with a gap of size $r$:
\[ j_1 - i_s \geq r.\]
By selecting a mapping $\Psi_{\eps}$ we obtain various coefficients of dependence.
\begin{align*}
\Psi_{\alpha}(f,g)&=4,  &\epsilon(r)&=\alpha(r) &\text{\ ($\alpha$-mixing).} \\
\Psi_{\theta}(f,g)&=t\,\Lip g, &\epsilon(r)&=\theta(r) &\text{\ ($\theta$-dependence).} \\
\Psi_{\eta}(f,g)&=s\, \Lip f+t\, \Lip g, & \epsilon(r)&=\eta(r) &\text{\ ($\eta$-dependence).} \\
\Psi_{\kappa}(f,g)&=s t \,\Lip f \cdot \Lip g, & \epsilon(r)&=\kappa(r) &\text{\ ($\kappa$-dependence).} \\
\Psi_{\lambda}(f,g)&=s\,\Lip f+t\,\Lip g+ st \,\Lip f\cdot\Lip g, & \epsilon(r)&=\lambda(r) &\text{\ ($\lambda$-dependence).}
\end{align*}

Notice that $\alpha$-mixing (often called also strong mixing) and $\theta$-dependence are {\em causal}, in the sense that they provide a bound for covariances with arbitrary measurable function $f$ of the past, while $\eta, \kappa$ or $\lambda-$dependencies are {\em non-causal}.

\subsection{Stationary $\alpha$-mixing sequences with a continuous marginal distribution}
\label{sect_cont_alpha}

\begin{theorem}\label{thm:alpha}
If $\{X_j\}$ is a stationary $\alpha$-mixing sequence with continuous mar\-gi\-nals,
then it admits a continuous phantom distribution function. 
\end{theorem}
\begin{proof} is given in Section \ref{proof:alpha}.
\end{proof}

\begin{example}\label{ex:MHA} (Random walk Metropolis algorithm with heavy-tailed marginals)

Let $\{Z_j\}$ is an i.i.d. sequence with the marginal distribution function $H$ given by the {\em proposal} density $h$, which is symmetric about $0$, and let $\{U_j\}$ be an i.i.d. sequence distributed uniformly
on $[0,1]$, independent of $\{Z_j\}$. Choose and fix the {\em target} probability density $f(x)$.
  
Let us consider a Markov chain given by the recursive equation
\begin{equation}\label{eq:MH}
X_{j+1} = X_j + Z_{j+1} \1\big\{ U_{j+1} \leq \psi\big(X_j, X_j + Z_{j+1}\big)\big\},
\end{equation}
where  $\psi(x,y)$ is defined as
\begin{equation}\label{eq:MH2}
\psi(x,y) = \begin{cases}
\min\big\{f(y)/f(x), 1\big\} &\text{ if } f(x) > 0,\\
1 & \text{ if } f(x) = 0.
\end{cases}
\end{equation}
Standard arguments based on the detailed balance equation $f(x)\psi(x,y) = f(y)\psi(y,x)$ and the assumed symmetry of $h$ show that $f$ is the
density of the stationary distribution function $F$ for $\{X_j\}$.  We refer to \cite{RRSS06} for discussion, references and a background relating such Markov chains to the well-known  random walk Metropolis algorithm and Markov Chain Monte Carlo methods. 

Here we focus on the problem of existence of a phantom distribution function for $\{X_j\}$. 
Since the marginal distribution function $F$ is continuous (when we run the process under the initial stationary distribution $F$) we can apply Theorem \ref{thm:alpha} provided we can verify $\alpha$-mixing of $\{X_j\}$. In the literature on Markov chains it is customary to assume that the chain is    ``$\psi$-irreducible and aperiodic" (see e.g. \cite{JaRo07}). But this is almost like assuming $\alpha$-mixing itself. Since the transition function for the random walk Metropolis algorithm is relatively simple we decided to provide {\em a particular and suitable for simulations} set of sufficient conditions imposed on $h$ and $f$, in order to convince the reader that any target density exhibiting minimum regularity leads to $\alpha$-mixing.

\begin{proposition}\label{Metropolis:mixing}
Suppose the proposal density $h$ and the target density $f$ satisfy the following conditions.
\begin{description}
\item{\em\bf (i)} The set $\GS = \{ x\in\GR\,;\, f(x) > 0\}$ is connected.
\item{\em\bf (ii)} 
 In some interval $[a,b]$, $a < b, a, b \in \GS$, $f$ is {\em monotone} and without intervals of constancy of length greater than $(b-a)/4$. 
\item{\em\bf (iii)} $h$ is symmetric around $0$ and for some $k_h > 0$
\begin{equation}\label{eq:h}  h(x) \geq k_h, \text{ if $|x| \leq (b-a)/3$}.
\end{equation}
\end{description}
Then $\{X_j\}$ is Harris recurrent and aperiodic, and, in particular, $\alpha$-mixing.
\end{proposition}

\begin{proof} First we will prove that (ii) implies strong aperiodicity.  Indeed,
\begin{align*}
\P\big( X_j = X_{j+1}\big) &= \P \big( U_{j+1} > \psi\big(X_j, X_j + Z_{j+1}\big)\big)\\
 & = \E\big( 1 - \psi\big(X_j, X_j + Z_{j+1}\big)\big)\\
&= \E\Big( 1 - \frac{f(X_j + Z_j)}{f(X_j)}\Big)\1\{ f(X_j + Z_j) < f(X_j)\} \1 \{ f(X_j) > 0\}.
\end{align*}
This expression will be positive if we are able to show that $ f(X_j + Z_j) < f(X_j)$ with positive probability.  Assume that $f$ is {\em non-increasing} on $[a,b]$ (the other case is completely analogous). Then  $f(a) \geq f(x) \geq f(b) > 0,\ x\in [a,b]$, hence
\begin{equation}\label{eq:mon}
\psi(x,y) \geq f(b)/f(a) > 0, \ x, y \in [a,b].
\end{equation}
Set for notational convenience $\eta = (b-a)/3$ and  observe that we have 
\begin{align*}
\P\big( f(X_j + Z_j) < f(X_j)\big) & = \int dx\, f(x)  \int dz\, h(z)\ \1\{ f(x + z) < f(x)\}\\
&\geq \int_a^{b - \eta} dx\, f(x)  \int_0^{\eta} dz\, h(z)\ \1\{ f(x + z) < f(x)\} \\
&\geq f(b) \frac{2}{3} (b-a) k_h \frac{1}{12}(b-a) = \frac{1}{18} f(b) k_h (b-a)^2 > 0.
\end{align*}

Next we shall choose a ``small'' set.  Set $C = [a+\eta, b-\eta]$. Notice that $x\in C$ and $z\in [-\eta, \eta]$ imply $x + z \in [a,b]$ and by (\ref{eq:mon}) 
\[ \psi(x,x+z) \geq f(b)/f(a).\]
Denote by $P(x, B)$ a regular version of the conditional distribution $\P\big( X_{n+1} \in B\big| X_n = x\big)$. We have for $B \subset C$ and $x \in C$ 
\begin{align*}
P(x, B) &= \delta_B(x) \int dz\, h(z) \big( 1 - \psi(x,x+z)\big) + \int dz\, h(z) \ \1_{B - x}(z) \psi(x,x+z) \\
&\geq \int dz\, h(z) \ \1_{B - x}(z) \psi(x,x+z) \geq k_h \int_{-\eta}^{\eta} dz\, \1_{B - x}(z) \psi(x,x+z)\\
&\geq k_h \big( f(b)/f(a)\big) \int_{-\eta}^{\eta} dz\, \1_{B - x}(z) = k_C \ell(B-x) = k_c \ell(B),
\end{align*}
where $\ell$ is the Lebesgue measure and the next-to-last equality holds because 
$B -x \subset C - C = [-\eta,\eta]$.

It is then a routine (although not straight-forward) application of Theorem 13.3.4 (ii) from 
\cite{MT09} (with $d=1$) or Theorems 21.5 and 21.6 from \cite{Brad07} that gives us $\alpha$-mixing of $\{X_j\}$.

\end{proof}

We have proved that any random walk Metropolis algorithm built upon functions $h$ and $f$ satisfying conditions (i)-(iii) of Proposition \ref{Metropolis:mixing} admits a continuous phantom distribution function.

It is interesting that in a wide class of target densities the extremal index of the corresponding 
Metropolis Markov chain is {\em zero}. Theorem 3.1 of \cite{RRSS06} asserts that this is the case when the target distribution function has the property that for some $m > 0$
\begin{equation}\label{lim1}
 \lim_{u\to\infty} \frac{1 - F(u + m)}{1 - F(u)}  = 1.
\end{equation}
\end{example}
 \begin{remark}\label{Rem:relative}
Random walk Metropolis algorithms are easy for simulation. Therefore their partial maxima can be used as a natural reference for processes which admit phantom distribution functions, but their extremal index is zero.  To be more precise, with each random walk Metropolis algorithm $\{X_j\}$ we can associate all stationary processes $\{X_j'\}$ such that for some $\theta \in (0,\infty)$ 
\[ \sup_{x\in\GR} \big|\P\big( \max_{1\leq j \leq n} X_j' \leq x\big) - \P^{\,\theta}\big( \max_{1\leq j \leq n} X_j \leq x\big)\big| \to 0,\ \text{ as $n\to\infty$}.\]
Following \cite{Jak91} we can say that $\{X_j'\}$ has the relative extremal index $\theta$ 
with respect to $\{X_j\}$ and we can investigate asymptotic properties of $\{M_n'\}$ through 
those of $\{M_n\}$.
\end{remark}

\begin{remark}\label{Rem:interpretation} Our random walk Metropolis algorithms admit  regenerative structures, similarly to the Lindley process. The difference is that it is rather hopeless task to provide a detailed description for the tail probabilities of the cycle maximum
$\{Y_j\}$. We know, however, by Theorem \ref{thm:reg} and Theorem \ref{thm:equiv} that $\P^{1/\mu}\big( Y_1 \leq x)$ must be strictly tail equivalent to the phantom distribution function obtained through our Theorem \ref{thm:alpha}. This distribution function can in turn be recovered form 
{\em the driving sequence} $\{v_n\}$ (such that $\lim_{n\to\infty} \P\big( M_n \leq v_n\big) = \gamma$, see (\ref{ephdf})). Thus if we are able to estimate 
the shape of the sequence $\{v_n\}$
we are also given some information on the tails of the cycle maximum.
Moreover, this relation brings some insight into the interpretation of the phantom distribution function when the extremal index is zero.
\end{remark}

\subsection{Concentration assumption and weak dependence}
The cases of other dependencies are not as simple as $\alpha$-mixing and require an additional assumption on the marginal distribution function $F$.\\[3mm]
{\bf Concentration assumption.} There exist constants $b> 0$ and $B> 0$ such that
\begin{equation}\label{Concentration}
 \P \big(X_1 \in (x,x+u]\big) = F(x+u) - F(x) \leq B u^b, \ x\in\GR, u > 0.
\end{equation}
\begin{remark}\label{remconc}
The concentration assumption is not very restrictive in the class of absolutely continuous distributions. For example, if  $F$ has a bounded density $p$, then (\ref{Concentration}) holds with $b=1$ and $B = \sup_x p(x)$. Another example is provided by the Beta density
\[ p(x) = \frac{x^{c-1}(1-x)^{d-1}}{B(c,d)}, \ x\in (0,1),\]
with $0 < c,d <1$. In this case $b = c \wedge d$. Notice that only $b \leq 1$ is possible.
\end{remark}

The following theorems are proved in Sections \ref{prooftheta} - \ref{proofkappa}, respectively.
\begin{theorem}\label{PHDFtheta}
If $\{X_j\}$ is a stationary sequence with continuous marginals satisfying (\ref{Concentration}), which is $\theta$-weakly dependent and fulfills
\[ \theta(r) = \cO(r^{-\beta}),\quad\text{ for some } \beta > \frac{1+\sqrt{5}}{2}\Big(1 + \frac{1}{b}\Big),\]
then it admits a continuous phantom distribution function. 
\end{theorem}

\begin{theorem}\label{PHDFeta}
If $\{X_j\}$ is a stationary sequence with continuous marginals satisfying (\ref{Concentration}), which is $\eta$-weakly dependent and fulfills
\[ \eta(r) = \cO(r^{-\beta}),\quad\text{ for some } \beta > 2 \Big(1 + \frac{1}{b}\Big),\]
then it admits a continuous phantom distribution function. 
\end{theorem}

\begin{theorem}\label{PHDFkappa}
If $\{X_j\}$ is a stationary sequence with continuous marginals satisfying (\ref{Concentration}), which is $\kappa$-weakly dependent and fulfills
\[ \kappa(r) = \cO(r^{-\beta}),\quad\text{ for some } \beta > \big(1+\sqrt{5}\big) \Big(1 + \frac{2}{b}\Big),\]
then it admits a continuous phantom distribution function. 
\end{theorem}

\begin{remark} An inspection of the proof of the above theorem shows 
that it remains true if we replace $\kappa$- with $\lambda$-weak dependence.
\end{remark}

\subsection{Discontinuous marginals}
\label{sect_disc}

Let us rewrite first a part of  the regularity condition  
(\ref{e2}) in the form of\\
{\bf Condition $\Delta_0$}
\begin{equation}\label{edeltazero}
\lim_{x\to F_*-} \frac{\Delta F(x)}{1 - F(x)} = 0,
\end{equation}
where $\Delta F(x) = F(x) - F(x-)$.
We will also need a stronger version,  defined for $\xi > 0$.\\
{\bf Condition $\Delta_{\xi}$}
\begin{equation}\label{edeltaeta}
\sup_{x < F_*-} \frac{\Delta F(x)}{(1 - F(x))^{1 + \xi}} \leq M_{F,\xi} < +\infty. 
\end{equation}

\begin{theorem}\label{th3b}
Let $\{X_j\}$ be a stationary, {$\alpha$-mixing} sequence with marginal distribution function $F$, which is continuous at $F_*$.  Then it admits a continuous phantom distribution function provided:
\begin{description}
\item{\em (i)} $\{X_j\}$ is $m$-dependent (i.e. $\alpha(m+1) = 0$) and $F$ satisfies $\Delta_0$; 
\item{\em (ii)} For some constants $C > 0$ and $\rho \in [0,1)$ we have  $\alpha(n) \leq C \rho^n$ and 
$F$ satisfies $\Delta_\xi$ for some $\xi > 0$;
\item{\em (iii)} For some constants $C > 0$ and $\beta > 0$ we have  $\alpha(n) \leq C n^{-\beta}$ and 
$F$ satisfies $\Delta_\xi$ for some $\xi> 1/\beta$.
\end{description}
\end{theorem} 

\begin{remark}\label{JaRo}
See \cite{JaRo07} for conditions giving the polynomial rate of $\alpha$-mixing 
in random walk Metropolis algorithms.
\end{remark}

\section{Proofs of Theorems \ref{thm:alpha} - \ref{th3b}}
\label{sect_proofs}

\subsection{Basic computations involving covariances}\label{Basic}

Let $\{X_j\}$ be a stationary sequence of real valued random variables
with a marginal distribution function $F$. 
Take  $\gamma \in (0,1)$ and  define
\begin{equation}\label{egam}
  v_n = \inf \{ x\,;\, \P( M_n \leq x) \geq \gamma\}.
\end{equation}
Clearly, $\{v_n\}$ is a {\em non-decreasing} sequence and we have 
\begin{equation}\label{egamma}
\P( M_n \leq v_n) \geq \gamma.
\end{equation}
 \begin{lemma} \label{ineq}
Set 
\begin{equation}\label{zetjotem}
Z_k(m)=\max\{X_m, X_{2m}, \dots, X_{km}\},\ k, m\in\GN.
\end{equation}
 If $k\cdot m \leq n$ then 
\begin{equation}\label{ein}
 \gamma \leq \P\big(M_n\le v_n\big)\le \P\big(X_1\le v_n\big)^k  + k C_n(m; k),
\end{equation}
where
\[ C_n(m; k)=\max_{2\le j\le k}\left| \P\big( Z_j(m) \leq v_n\big) - \P\big(X_1\le v_n\big)\P\big( Z_{j-1}(m) \leq v_n\big)\right|.\]
\end{lemma}
\begin{proof} If $k\cdot m \leq n$,  then
\begin{align*}\gamma \leq&\ \P\big( M_n \leq v_n\big) \leq \P\big( M_{k m} \leq v_n\big)
\leq \P\big( Z_k(m) \leq v_n\big) \\  
\leq&\  \P\big(X_{km}\le v_n\big) \P\big( Z_{k-1}(m) \leq v_n\big)\\ & + \left| \P\big( Z_k(m) \leq v_n\big) - \P\big(X_{km}\le v_n\big)\P\big( Z_{k-1}(m) \leq v_n\big)\right|\\ 
 \leq&\ \P\big(X_1\le v_n\big)^2\P\big( Z_{k-2}(m) \leq v_n\big) + C_n(m; k)\\
  &+ \P\big(X_1\le v_n\big)\left| \P\big( Z_{k-1}(m) \leq v_n\big) - \P\big(X_{(k-1)m}\le v_n\big)\P\big( Z_{k-2}(m) \leq v_n\big)\right| \\
\leq&\ \P\big(X_1\le v_n\big)^2 \P\big( Z_{k-2}(m) \leq v_n\big) + 2 C_n(m; k) \\
\leq&\ \cdots \leq  \P\big(X_1\le v_n\big)^k  + k C_n(m; k).
 \end{align*}
This concludes the proof. \end{proof}
\begin{proposition}\label{Propbasic}
If $k_n \to \infty$ and $m_n \in \GN$ is such that $k_n\cdot m_n \leq n$ and 
\begin{equation}\label{eqmix}
k_n C_n(m_n; k_n) \to 0, \text{ as $n\to \infty$},
\end{equation}
then
\begin{equation}\label{eqbound}
 \sup_n k_n \P\big( X_1 > v_n\big) < +\infty.
\end{equation}
\end{proposition}
\begin{proof}
For sufficiently large $n$ we have $k_n C_n(m_n; k_n) \leq \gamma/2$. For such $n$ relation (\ref{ein}) gives 
\[ \gamma/2 \leq \P\big(X_1\le v_n\big)^{k_n}.\]
Since $\{v_n\}$ is a non-decreasing sequence, $c = \lim_{n\to\infty} \P\big( X_1 \leq v_n\big)$ exists, and $\gamma/2 \leq c^{k_n}, n\in \GN$. It follows that $c = 1$ and for sufficiently large $n$
\[ k_n \P\big(X_1 > v_n\big) \sim - k_n \ln \big(\P\big(X_1\le v_n\big)\big) \leq \ln3 - \ln\gamma.\] 
\end{proof}

\subsection{Checking Condition $B_{T}(v_n)$}\label{CheckingBT}

\begin{proposition}\label{lembeeren}
Let $\{X_j\}$ be a stationary sequence and $\{v_n\}$ be a sequence of levels.
Suppose  $\{r_n\}$ is such that 
\begin{equation}\label{eqeren}
r_n \P\big( X_1 > v_n\big) \longrightarrow 0.
\end{equation}
Then Condition $B_T(v_n)$ holds iff for all sequences $p_n > r_n$ and $q_n$, $p_n + q_n \leq T\cdot n$,
\begin{equation}\label{condbete}
\mathbb{C}\text{\em ov}\big( h_n(M_{p_n - r_n}), h_n(M_{p_n:p_n + q_n})\big) \longrightarrow 0,
\end{equation}
where $h_n(x) = \1( x \leq v_n)$ and $M_{p:q} = \max_{p< j \leq q} X_j$. 
\end{proposition}
\begin{proof}
Condition $B_{T}(v_n)$ is equivalent to the statement that for arbitrary sequences 
$\{p_n\}, \{q_n\} \subset \GN$ satisfying $p_n + q_n \leq T\cdot n$ we have 
\begin{equation}\label{condbe}
 \P\big( M_{p_n + q_n} \leq v_n\big)  =  \P\big( M_{p_n} \leq v_n\big) 
\ P\big( M_{q_n} \leq v_n\big) + o(1).
\end{equation}
Let $\{r_n\}$ satisfies (\ref{eqeren}).
First let us prove that (\ref{condbete}) implies Condition $B_T(v_n)$. 

Take any $\{p_n\}$ and $\{q_n\}$ satisfying $p_n + q_n \leq T\cdot n$. 
Suppose that $p_{n'} \leq r_{n'}$ along a subsequence $\{n'\}$. Then we have 
$\P\big( M_{p_{n'}} \leq v_{n'} \big) \geq \P\big( M_{r_{n'}} \leq v_{n'} \big) \to 1$ as well as 
\begin{align*}
 0 \leq \P\big( M_{q_{n'}} \leq v_{n'}\big)  &- \P\big( M_{p_{n'} + q_{n'}} \leq v_{n'}\big) \leq
\P\big( M_{p_{n'}} > v_{n'} \big) \\
&\leq    \P\big( M_{r_{n'}} > v_{n'} \big) \leq r_{n'} \P\big( X_1 > v_{n'} \big)
\to 0,
\end{align*}
And so
\begin{align*}
 \P\big( M_{p_{n'} + q_{n'}} \leq v_{n'}\big)  &=  \P\big( M_{q_{n'}} \leq v_{n'}\big) + o(1) \\ 
 &=\P\big( M_{p_{n'}} \leq v_{n'}\big) \P\big( M_{q_{n'}} \leq v_{n'}\big) + o(1).
\end{align*}
Hence along the subsequence $\{n'\}$ (\ref{condbe}) holds. So we may and do assume that $p_n > r_n,\ n\in\GN$. Similarly as above we obtain
\begin{align}
\P\big(M_{p_n} \leq v_n\big) &= \P\big(M_{p_n - r_n} \leq v_n\big) + o(1).
\label{quenquen}\\
\P\big(M_{p_n + q_n} \leq v_n\big) &= \P\big(M_{p_n - r_n} \leq v_n, 
M_{p_n : p_n + q_n} \leq v_n\big) + o(1). \label{peenpen}
\end{align}
We have
\begin{align*}
 \P&\big( M_{p_n + q_n} \leq v_n\big)  -  \P\big( M_{p_n} \leq v_n\big) 
\ P\big( M_{q_n} \leq v_n\big) \\
&= \P\big(M_{p_n - r_n} \leq v_n, 
M_{p_n : p_n + q_n} \leq v_n\big) - \P\big(M_{p_n -r_n} \leq v_n\big)
\P\big(M_{q_n} \leq v_n\big) + o(1) \\
&= \cov\big( h_n(M_{p_n - r_n}), h_n(M_{p_n : p_n + q_n})\big) \longrightarrow 0. \tag{by (\ref{condbete})}
\end{align*}
Thus under (\ref{eqeren}) Condition $B_T(v_n)$ is implied by (\ref{condbete}).

To prove the converse implication take any $p_n > r_n$ and $q_n$ such that  $p_n + q_n \leq T\cdot n$ and observe that by (\ref{quenquen}) and (\ref{peenpen})
\begin{align*}
\cov\big( h_n(M_{p_n - r_n}), &h_n(M_{p_n : p_n + q_n})\big) \\
&= 
\cov\big( h_n(M_{p_n}), h_n(M_{p_n : p_n + q_n})\big) + o(1)  
\longrightarrow 0. \tag{by (\ref{condbe})}
\end{align*}
\end{proof}

\begin{corollary}\label{alphamix}
If $\{X_j\}$ is stationary and $\alpha$-mixing and $\{v_n\}$ is such that $\P\big(X_1 > v_n\big) \to 0$, then Condition $B_{\infty}(v_n)$ holds.
\end{corollary}
\begin{proof} Take any $\{r_n\} \subset \GN,\ r_n\to\infty$ such that $r_n \P\big(X_1 > v_n\big) \to 0$. Then
\[\big|\cov\big( h_n(M_{p_n - r_n}), h_n(M_{p_n:p_n + q_n})\big)\big| \leq \alpha(r_n) \longrightarrow 0,\]
\end{proof}

\begin{remark}
Throughout this section the  levels $\{v_n\}$ are given by formula
(\ref{egamma}). For such sequences $\{v_n\}$ there is no {\em a priori} reason for $\P\big( X_1 > v_n\big) \to 0$. Moreover, it is also important to know {\em how fast} $\P\big( X_1 > v_n\big) \to 0$. Proposition \ref{Propbasic} shows how to answer both questions if we are given estimates for covariances of suitable functionals of $\{X_j\}$.  

The next result demonstrates how Proposition \ref{Propbasic} works in the simplest case of $\alpha$-mixing. 
\end{remark}

\begin{proposition}\label{propalpha}
If $\{X_j\}$ is stationary and $\alpha$-mixing and $\{v_n\}$ are defined by (\ref{egam}), then $k_n\alpha (m_n) \to 0$ implies 
\[ \sup_n k_n \P\big( X_1 > v_n\big) < +\infty.\] 
In particular, if $\alpha(r) = \cO(r^{-\beta})$, for some $\beta >0$, then for every $\delta < \beta/(1 + \beta)$ we have 
\begin{equation}\label{alphatwo}
\lim_{n\to\infty} n^{\delta} \P\big( X_1 > v_n\big) = 0.
\end{equation}
\end{proposition}
\begin{proof} By stationarity, the very definition of $\alpha$-mixing and with $Z_j(m)$ defined by (\ref{zetjotem})
\begin{align*}
\Big| \P\big( Z_j(m) &\leq v_n\big) - \P\big(X_1\le v_n\big)\P\big( Z_{j-1}(m) \leq v_n\big)\Big| = \\
& = \Big| \P\big( Z_{j-1}(m) \leq v_n, X_{jm} \leq v_n\big) - \P\big(X_{jm}\le v_n\big)\P\big( Z_{j-1}(m) \leq v_n\big)\Big| \\
&\leq \alpha(m).
\end{align*}
It follows that $C_n(m; k) \leq \alpha(m)$ and by Proposition \ref{Propbasic} $k_n\alpha(m_n) \to 0$ implies 
the boundedness of $ k_n \P\big( X_1 > v_n\big)$.  

Now suppose that $0 < \delta < \beta/(1+\beta)$. Choose  $\delta'$ satisfying $\delta < \delta' < \beta/(1+\beta)$. Set $k_n = [ n^{\delta'}] \sim   n^{\delta'}$ and $m_n = [n/k_n] \sim n^{1-\delta'}$. Then 
\[ k_n \alpha(m_n) = \cO \big( n^{\delta'} (n^{1-\delta'})^{-\beta}\big) = \cO\big( n^{\delta'(1+\beta) - \beta}\big) \longrightarrow 0.\]
Hence $\sup_n n^{\delta'} \P\big( X_1 > v_n\big) < +\infty$ and so (\ref{alphatwo}) holds.
\end{proof}

\begin{remark}
$\alpha$-mixing gives direct estimates for both the covariances in (\ref{condbete}) and the quantities $C_n(m,k)$ in Proposition \ref{Propbasic}. The other coefficients of weak dependence defined in Section \ref{coefficients} 
provide estimates for smooth (Lipschitz) functionals of $\{X_j\}$. Therefore, similarly as in the proof in Lemma 4.1 page 68 in \cite{DDL07}, we will consider a natural $1/u-$Lipschitz approximation $h_{n,u}$ of the function $h_n(x)= \1_{\{x\leq v_n\}}$, 
which is given by 
\begin{equation}
 h_{n,u}(x) = \begin{cases} 1 &\text{ if } x \leq v_n,\\
-(1/u) (x - v_n) + 1 & \text{ if } v_n < x \leq v_n+ u, \\
0 & \text{ if } x > v_n + u.
\end{cases}
\end{equation}
We have then two basic estimates for $C_n(m,k)$.
\end{remark}
\begin{lemma}\label{covtwo}
Let $\{X_j\}$ be stationary and $F$ satisfies (\ref{Concentration}) with some constants $B,b >0$.
Set $Y_{n,j} = h_n(X_{m})\cdots h_n(X_{(j-1)m}) = \1\big( Z_{j-1}(m) \leq v_n\big)$, where $Z_j(m)$ is defined by (\ref{zetjotem}). Then
\[ C_n(m,k) \leq  \max_{2 \leq  j \leq k}  \Big|\mathbb{C}\text{\em ov}\big(Y_{n,j}, h_{n,u}(X_{jm})\big)\Big| +  B u^b .\]
\end{lemma}
\begin{proof} We have
\begin{align*}
 \Big| \P\big( Z_j(m) \leq &v_n\big) - \P\big(X_{jm}\le v_n\big)\P\big( Z_{j-1}(m) \leq v_n\big)\Big| = \Big|\cov\big( Y_{n,j}, h_n(X_{jm})\big)\Big| \\
& \leq \Big|\cov\big(Y_{n,j}, h_{n,u}(X_{jm})\big)\Big|+\Big|\cov\big(Y_{n,j}, h_n(X_{jm})-h_{n,u}(X_{jm})\big)\Big|\\
& \leq  \Big|\cov\big(Y_{n,j}, h_{n,u}(X_{jm})\big)\Big|+ \E\big|h_n(X_{jm})-h_{n,u}(X_{jm})\big)\Big| \\
& \leq \Big|\cov\big(Y_{n,j}, h_{n,u}(X_{jm})\big)\Big| +  \big(F(v_n + u) - F(v_n)\big) \\
& \leq  \Big|\cov\big(Y_{n,j}, h_{n,u}(X_{jm})\big)\Big| + B u^b.
\end{align*}
\end{proof}
\begin{lemma}\label{covtwobis}
Let $\{X_j\}$ be stationary and $F$ satisfies (\ref{Concentration}) with some constants $B,b >0$. Then
\begin{align*}
\Big|\mathbb{C}\text{\em ov}&\big( h_n(M_{p_n - r_n}), h_n(M_{p_n:p_n + q_n})\big) \Big| \\
&\leq \Big|\mathbb{C}\text{\em ov}\big( h_n(M_{p_n - r_n}), h_{n,u}(M_{p_n:p_n + q_n})\big) \Big| + B q_n u^b.
\end{align*}
\end{lemma}
\begin{proof}
Similarly as before
\begin{align*}
 \Big|\cov&\big( h_n(M_{p_n - r_n}), h_n(M_{p_n:p_n + q_n})\big) \Big|  \leq 
\Big|\cov\big( h_n(M_{p_n - r_n}), h_{n,u}(M_{p_n:p_n + q_n})\big) \Big| \\
& +
\Big|\cov\big( h_n(M_{p_n - r_n}), h_n(M_{p_n:p_n + q_n})-  h_{n,u}(M_{p_n:p_n + q_n})\big)\big) \Big|\\
&\leq \Big|\cov\big( h_n(M_{p_n - r_n}), h_{n,u}(M_{p_n:p_n + q_n})\big) \Big| +  \P\big( v_n < M_{p_n:p_n + q_n}\leq v_n + u\big)\\
&\leq \Big|\cov\big( h_n(M_{p_n - r_n}), h_{n,u}(M_{p_n:p_n + q_n})\big) \Big| + q_nB u^b.
\end{align*}
\end{proof}

\begin{lemma}\label{covthree}
In assumptions and notations of Lemma \ref{covtwo} we have
\[ C_n(m,k) \leq  \max_{2 \leq j \leq k}  \Big|\mathbb{C}\text{\em ov}\big(h_{n,u}(Z_{j-1}(m)), h_{n,u}(X_{jm})\big)\Big| +  B k u^b .\]
\end{lemma}
\begin{proof}
\begin{align*}
\big|\cov&\big( Y_{n,j},\,  h_n(X_{jm})\big)\big|  \\
 & \leq \big|\cov\big(Y_{n,j}, h_{n,u}(X_{jm})\big)\big|+\big|\cov\big(Y_{n,j}, h_n(X_{jm})-h_{n,u}(X_{jm})\big)\big| \\
& \leq \big|\cov\big( Y_{n,j} - h_{n,u}(Z_{j-1}(m)), h_{n,u}(X_{jm})\big) \big| \\
&\qquad + \big|\cov \big( h_{n,u}(Z_{j-1}(m)), h_{n,u}(X_{jm})\big)\big| + \E\big|h_n(X_{jm})-h_{n,u}(X_{jm})\big)\Big| \\
& \leq  \E\big| Y_{n,j} - h_{n,u}(Z_{j-1}(m))\big| +  \big|\cov \big( h_{n,u}(Z_{j-1}(m)), h_{n,u}(X_{jm})\big)\big| +  B u^b \\
&\leq  \P\big( v_n < Z_{j-1}(m) \leq v_n + u\big) \\
&\qquad +\big|\cov \big( h_{n,u}(Z_{j-1}(m)), h_{n,u}(X_{jm})\big)\big| +  B u^b \\
& \leq  (j-1) \P\big( v_n < X_1 \leq v_n + u\big) \\
&\qquad +\big|\cov \big( h_{n,u}(Z_{j-1}(m)), h_{n,u}(X_{jm})\big)\big| +  B u^b \nonumber\\
&\leq \big|\cov \big( h_{n,u}(Z_{j-1}(m)), h_{n,u}(X_{jm})\big)\big| +   j B u^b.
\end{align*}
\end{proof}

\begin{lemma}\label{covthreebis}
Let $\{X_j\}$ be stationary and $F$ satisfies (\ref{Concentration}) with some constants $B,b >0$. Then
\begin{align*}
\Big|\mathbb{C}\text{\em ov}&\big( h_n(M_{p_n - r_n}), h_n(M_{p_n:p_n + q_n})\big) \Big| \\
&\leq \Big|\mathbb{C}\text{\em ov}\big( h_{n,u}(M_{p_n - r_n}), h_{n,u}(M_{p_n:p_n + q_n})\big) \Big| + B (q_n + p_n)u^b.
\end{align*}
\end{lemma}

\begin{proof}
\begin{align*}
 \Big|\cov\big( h_n&(M_{p_n - r_n}),h_n(M_{p_n:p_n + q_n})\big) \Big| \\
&\leq 
\Big|\cov\big( h_n(M_{p_n - r_n}), h_n(M_{p_n:p_n + q_n})-  h_{n,u}(M_{p_n:p_n + q_n})\big)\big) \Big|\\
&\qquad  + \Big|\cov\big( h_{n}(M_{p_n - r_n}) -  h_{n,u}(M_{p_n - r_n}) , h_{n,u}(M_{p_n:p_n + q_n})\big) \Big| \\
&\qquad + 
\Big|\cov\big( h_{n,u}(M_{p_n - r_n}), h_{n,u}(M_{p_n:p_n + q_n})\big) \Big| \\
&\leq    \P\big( v_n < M_{p_n:p_n + q_n}\leq v_n + u\big) + \P\big( v_n < M_{p_n-r_n}\leq v_n + u\big) \\
&\qquad  + \Big|\cov\big( h_{n,u}(M_{p_n - r_n}), h_{n,u}(M_{p_n:p_n + q_n})\big) \Big|\\
&\leq  (p_n + q_n) B u^b + \Big|\cov\big( h_{n,u}(M_{p_n - r_n}), h_{n,u}(M_{p_n:p_n + q_n})\big) \Big|.
\end{align*}
\end{proof}

\subsection{Proof of Theorem \ref{thm:alpha}}\label{proof:alpha}
Suppose that $\{X_j\}$ has {\em continuous} marginal distribution function $F$. 
 Let us define, as in (\ref{egam}), $v_n = \inf \{ x\,;\, \P\big( M_n \leq x\big) \geq \gamma\}$, for some $\gamma \in (0,1)$. Since
$F(x)$ is continuous, so is the distribution function of 
$M_n$, for each $n \in \GN$. Hence
\begin{equation}\label{ehalf}
\P\big( M_n \leq v_n\big) = \gamma,\ n\in\GN,
\end{equation}
and  relation (\ref{e3}) holds. Thus in view of Theorem \ref{th1} (iii) or (iv), if we want to 
find a phantom distribution function for $\{X_j\}$ we have to verify either Condition $B_{\infty}(v_n)$ or Condition $B_T(v_n)$, for each $T > 0$.

Suppose that $\alpha(r) \to 0$, as $r\to\infty$, and let $k_n \to\infty$ be such that $k_n\leq \sqrt{n}$ and $k_n\alpha(\sqrt{n}) \to 0$. Then by Proposition \ref{propalpha} $\sup_n k_n
\P\big(X_1 > v_n\big) < +\infty$, and so $\P\big(X_1 > v_n\big) \to 0$. Hence to check Condition $B_{\infty}(v_n)$ we can apply Corollary  \ref{alphamix}.

\subsection{Proof of Theorem \ref{PHDFtheta}}\label{prooftheta}
Since $F$ is continuous by the concentration assumption  (\ref{Concentration}), we can  define $v_n$ by (\ref{egam}) and then (\ref{ehalf}) holds. So it remains to verify Condition $B_T(v_n)$ for any $T >0$ or by Proposition \ref{lembeeren} to check (\ref{condbete}) with $\{r_n\}$ such that  $r_n \P\big( X_1 > v_n\big) \to 0$.

The coefficient $\theta$ is causal, so Lemmas \ref{covtwo} and \ref{covtwobis} are applicable. By the latter and the very definition of $\theta$
we have
\begin{align*}
\Big|\cov&\big( h_n(M_{p_n - r_n}), h_n(M_{p_n:p_n + q_n})\big) \Big| \\
&\leq \Big|\cov\big( h_n(M_{p_n - r_n}), h_{n,u}(M_{p_n:p_n + q_n})\big) \Big| + B q_n u^b \\
&\leq q_n\big(\frac{1}{u} \theta(r_n) + B u^b\big) \leq T\cdot n \big(\frac{1}{u} \theta(r_n) + B u^b\big) \\
&= T(1 + B)\cdot n\cdot \theta^{b/(1+b)}(r_n),
\end{align*}
where in the last equality we substituted $u = \theta^{1/(1+b)}(r_n)$. 

We look for $r_n$ of the form $[n^{\tau}] \sim n^{\tau}$, for some $0 < \tau < 1$. From the above estimate we need $n \cdot \theta^{b/(1+b)}(r_n) = \cO( n^{1 - \tau\cdot \beta b/(1+b)}) \to 0$ or 
\[(1+b)/(b\beta) < \tau.\]
We have to check whether for some $\tau$ satisfying the above bound $n^{\tau} \P\big( X_1 > v_n\big) \to 0$. By Proposition \ref{Propbasic} 
it is enough to check whether for some $1 -\tau > \delta > 0$  
\[ n^{\tau + \delta} C_n\big( [n^{1-\tau-\delta}]; [n^{\tau + \delta}]) \to 0.\]
By Lemma \ref{covtwo} and the trick with substitution $u = \theta^{1/(1+b)}(m)$ we have
\[ C_n(m ; k) \leq \frac{1}{u} \theta(m) + B u^b  = (1 + B) \theta^{b/(1+b)}(m).\]
So in this case we need $ n^{\tau + \delta} \big(n^{1 - \tau - \delta}\big)^{-\beta(b/(1+b))} = n^{\big((\tau +\delta)(1 + b + \beta b) - \beta b\big)/(1+ b)} \to 0$ or
\[ \tau + \delta < \frac{\beta b}{1 + b + \beta b}.\]
For $b$ fixed, $\tau> 0$ satisfying both required bounds does exist iff  $\beta > (1 + \frac{1}{b}) \varphi$, where 
\[ \varphi = \frac{1+\sqrt{5}}{2}  \approx 1.618\ldots \]
is the {\em Golden Ratio}.

\subsection{Proof of Theorem \ref{PHDFeta}}\label{proofeta}

First let us observe that  $h_{n,u}( x_1\vee x_2 \vee \cdots \vee x_j)$ is a $1/u$-Lipschitz approximation of  $h_{n}( x_1\vee x_2 \vee \cdots \vee x_j) = \1\big(
x_1 \leq v_n, x_2 \leq v_n, \ldots, x_j \leq v_n\big)$. Indeed, for each $j\ge1$ the function $m_j:\,(x_1,\ldots,x_j)\mapsto  x_1 \vee x_2 \vee \cdots \vee x_j$ is $1-$Lipschitz in the sense that $|m_j(x)-m_j(y)|\le |x_1-y_1|+\cdots+ |x_j-y_j|$. Hence we have 
$\Lip (h_{n,u}\circ m_j)=\frac1u$.

Next we follow the proof of Theorem \ref{PHDFtheta}, using Lemmas \ref{covthree} and \ref{covthreebis} due to the non-causality of the coefficient $\eta$. 
We have by Lemma \ref{covthreebis}, the definition of $\eta (r) $ and the trick with substitution $u = \eta^{1/(1+b)}(r_n)$
\begin{align*}
\Big|\cov&\big( h_n(M_{p_n - r_n}), h_n(M_{p_n:p_n + q_n})\big) \Big| \\
&\leq \Big|\cov\big( h_{n,u}(M_{p_n - r_n}), h_{n,u}(M_{p_n:p_n + q_n})\big) \Big| + B (q_n + p_n)u^b\\
&\leq \frac{p_n + q_n}{u} \eta(r_n) + B (q_n + p_n)u^b\\
&\leq T \cdot n \big( \frac{\eta(r_n)}{u} + B u^b\big) = T(1+B)\cdot n\cdot \eta^{b/(1+b)}(r_n).
\end{align*}
Setting $r_n = [n^{\tau}] \sim n^{\tau}$, for some $0 < \tau < 1$, we need, as for $\theta$-dependence, $(1+b)/(b\beta) < \tau $.

By Lemma \ref{covthree} and substituting $u = \eta^{1/(1+b)}(m)$ we have
\begin{align*}
C_n(m,k) &\leq  \max_{ 2\leq j \leq k}  \Big|\cov\big(h_{n,u}(Z_{j-1}(m)), h_{n,u}(X_{jm})\big)\Big| +  B k u^b\\
&\leq \max_{2 \leq j \leq k}  \frac{(j-1) + 1}{u} \eta(m) + B k u^b \\
& = k \big(\frac{\eta(m)}{u} + B u^b\big) = (1+B) k \eta^{b/(1+b)}(m).
\end{align*}
By Proposition \ref{Propbasic} we need for some $1 - \tau > \delta > 0$
\[n^{\tau + \delta} n^{\tau + \delta} (n^{1-\tau -\delta}\big)^{-\beta b/(1+b)} \to 0\] or
$\tau + \delta < \beta b/(2 + 2b + \beta b)$. For $b$ fixed, $\tau> 0$ satisfying both required bounds does exist iff  $\beta > 2(1 + \frac{1}{b})$.

\subsection{Proof of Theorem \ref{PHDFkappa}}\label{proofkappa}

We follow the proof of Theorem \ref{PHDFeta}. By Lemma \ref{covthreebis}, the definition of $\kappa (r) $ and the substitution $u = \eta^{1/(2+b)}(r_n)$ we have 
\begin{align*}
\Big|\cov&\big( h_n(M_{p_n - r_n}), h_n(M_{p_n:p_n + q_n})\big) \Big| \\
&\leq \Big|\cov\big( h_{n,u}(M_{p_n - r_n}), h_{n,u}(M_{p_n:p_n + q_n})\big) \Big| + B (q_n + p_n)u^b\\
&\leq \frac{p_n \cdot q_n}{u^2} \kappa(r_n) + B (q_n + p_n)u^b\\
&\leq T^2 \cdot n^2 \big( \frac{\kappa(r_n)}{u^2} + B u^b\big) = T^2(1+B) \cdot n^2 \cdot \kappa^{b/(2+b)}(r_n).
\end{align*}
Setting $r_n = [n^{\tau}] \sim n^{\tau}$, for some $0 < \tau < 1$, we need
$n^2 \cdot \kappa^{b/(2+b)}(r_n) = \cO( n^{2 - \tau\cdot \beta b/(2+b)}) \to 0$ or  $2(2+b)/(b\beta) < \tau $.

By Lemma \ref{covthree} and substituting $u = \kappa^{1/(2+b)}(m)$ we have
\begin{align*}
C_n(m,k) &\leq  \max_{ 2\leq j \leq k}  \Big|\cov\big(h_{n,u}(Z_{j-1}(m)), h_{n,u}(X_{jm})\big)\Big| +  B k u^b\\
&\leq \max_{2 \leq j \leq k}  \frac{(j-1)\cdot 1}{u^2} \kappa(m) + B k u^b \\
& = k \big(\frac{\kappa(m)}{u^2} + B u^b\big) = (1+B) k \kappa^{b/(2+b)}(m).
\end{align*}

By Proposition \ref{Propbasic} we need for some $1 - \tau > \delta > 0$
\[n^{\tau + \delta} n^{\tau + \delta} (n^{1-\tau -\delta}\big)^{-\beta b/(2+b)} \to 0\] or
$\tau + \delta < \beta b/(4 + 2b + \beta b)$. For $b$ fixed required
 $\tau> 0$ exists iff  $\beta > 2 \varphi (1 + \frac{2}{b})$.

\subsection{Proof of Theorem \ref{th3b}}

Let us assume that the marginal distribution function $F$ allows jumps. Then the distribution of $M_n$ is also discontinuous. Hence defining $\{v_n\}$ through 
(\ref{egam}) we have only 
\begin{equation}\label{ehalfmore}
 \P\big( M_n \leq v_n \big) \geq \gamma,
\end{equation}
and we cannot directly conclude that $\P\big( M_n \leq v_n\big) \to \gamma$. 

On the other hand, (\ref{ehalfmore}) is enough for all results of Sections \ref{Basic} and \ref{CheckingBT} to hold. In particular, following the proof of Theorem \ref{thm:alpha}, we obtain from Proposition \ref{propalpha}
and Corollary \ref{alphamix} that  for $\alpha$-mixing sequences Condition $B_{\infty}(v_n)$ is satisfied. 

To guarantee  that $\P\big( M_n \leq v_n\big) \to \gamma$  (equivalently: $\P\big( M_n = v_n\big) \to 0$),
we need additional assumptions.

By Proposition \ref{propalpha} we have that
\[ k_n\P\big(X_1 > v_n\big) < +\infty,\]
whenever
$k_n \alpha(m_n) \to 0$ and $k_n \to \infty$, $k_n m_n \leq n$.
In particular, if $\alpha(m+1) = 0$, then 
\[ \sup_{n} [n/(m+1)] \P\big( X_1 > v_n) = L_1< +\infty.\]
If $F$ satisfies $\Delta_0$, then
\begin{align*}
\P\big( M_n = v_n\big) &\leq n \P\big( X_1 = v_n\big) 
= n \P\big( X_1 > v_n\big) \frac{\P\big( X_1 = v_n\big)}{\P\big( X_1 > v_n\big)} \\
& \leq (m+1) L_1 \frac{\P\big( X_1 = v_n\big)}{\P\big( X_1 > v_n\big)}\to 0.
\end{align*}

Next let us notice that (ii) is implied by (iii), for if $\xi$ is fixed, then the exponential rate of mixing implies a polynomial rate with arbitrary $\beta > 1/\xi$.  

So suppose that $\alpha(n) \leq C n^{-\beta}$, $F$ satisfies $\Delta_{\xi}$ and $1 < \xi  \beta.$
Then $1 + \xi < \xi \beta + \xi$ and so $\frac{1}{1+\beta} < \frac{\xi}{1+\xi}.$
Let us choose $\delta$ such that
\begin{equation}\label{goodes}
\frac{1}{1+\beta} < \delta < \frac{\xi}{1+\xi}.
\end{equation}
Let $m_n = [ n^\delta ]$ and let $k_n = [ n/m_n ]$. Then we check that due to $\delta > 1/(1+\beta)$ that
\[ k_n \alpha (m_n) \leq C k_n {m_n}^{-\beta} \sim  C n^{(1-\delta)} n^{-\delta\beta} \to 0.\]
Hence by Proposition \ref{propalpha}
\[ \sup_n n^{1-\delta} \P\big( X_1 > v_n\big) \leq L_2 <+\infty.\]
Since $\delta < \xi/(1+\xi)$ is equivalent to 
$\delta/\xi < (1-\delta)$, we obtain 
\begin{align*}
\P\big( M_n = v_n\big) &\leq n \P\big( X_1 = v_n\big) 
=  \frac{\P\big( X_1 = v_n\big)}{\P\big( X_1 > v_n\big)^{1 + \xi}} n \P\big( X_1 > v_n\big)^{1 + \xi} \\
& \leq M_{F,\xi} n^{1-\delta} \P\big( X_1 > v_n\big)   n^{\delta}\P\big( X_1 > v_n\big)^{\xi} \\
& \leq M_{F,\xi} L_2  \Big(n^{\delta/\xi}\P\big( X_1 > v_n\big)\Big)^{\xi} \to 0, \text{\qquad since $n^{\delta/\xi} = o(n^{1-\delta})$.}
\end{align*}

\begin{acknowledgements}
We would like to thank the Associated Editor and the Referee for their constructive 
criticism and for turning our attention to the important paper by \cite{Asmu98}.
\end{acknowledgements}
\bibliographystyle{aps-nameyear}     

\end{document}